\documentclass[12pt,a4paper]{amsart}
\usepackage{amsfonts}
\usepackage{amsthm}
\usepackage{amsmath, amscd, amssymb, mathtools, color}
\usepackage[latin2]{inputenc}
\usepackage{t1enc}
\usepackage[mathscr]{eucal}
\usepackage{indentfirst}
\usepackage{graphicx}
\usepackage{graphics, tikz}
\usepackage{hyperref}
\usepackage{pict2e}
\usepackage{epic}
\numberwithin{equation}{section}
\usepackage[margin=2.9cm]{geometry}
\usepackage{epstopdf}

\DeclareMathOperator{\reg}{reg}
\DeclareMathOperator{\lex}{lex}

\def\NZQ{\Bbb}               % the font for N,Z,Q,R,C

\def\ZZ{{\NZQ Z}}

\def\frk{\frak}               % font for "Fraktur"

\def\Phi{{\frk n}}
\def\Phi{{\frk N}}
\def\MR{{\mathcal R}}

\def\root {{>_\MR}}
\def\opn#1#2{\def#1{\operatorname{#2}}} % to make operators
\opn\div{div} \opn\Div{Div} \opn\cl{cl} \opn\Cl{Cl}
 \theoremstyle{plain}
\newtheorem{Theorem}{Theorem}[section]
 \newtheorem{Lemma}[Theorem]{Lemma}
 \newtheorem{Corollary}[Theorem]{Corollary}
 \newtheorem{Proposition}[Theorem]{Proposition}
 
 \newtheorem{conjecture}[Theorem]{Conjecture}

 \theoremstyle{remark}

 \newtheorem{Definition}[Theorem]{Definition}
 \newtheorem{Remark}[Theorem]{Remark}
 
 \newtheorem{Example}[Theorem]{Example}
 
 \newtheorem{Notation}[Theorem]{Notation}
 \newtheorem{question}[Theorem]{Question}

\let\epsilon\varepsilon
\let\kappa=\varkappa
\begin{document}

\title{Rooted order on minimal generators of powers of some cover ideals}

\author[N. Erey]{Nursel Erey}

\address{Gebze Technical University \\ Department of Mathematics \\
Gebze \\ Kocaeli \\ 41400 \\ Turkey} 

\email{nurselerey@gtu.edu.tr}



\begin{abstract} 
 We define a total order, which we call rooted order, on minimal generating set of $J(P_n)^s$ where $J(P_n)$ is the cover ideal of a path graph on $n$ vertices. We show that each power of a cover ideal of a path has linear quotients with respect to the rooted order. Along the way, we characterize minimal generating set of $J(P_n)^s$ for $s\geq 3$ in terms of minimal generating set of $J(P_n)^2$. We also discuss the extension of the concept of rooted order to chordal graphs. Computational examples suggest that such order gives linear quotients for powers of cover ideals of chordal graphs as well.
\end{abstract}

\maketitle

\section{Introduction}
Let $S=\Bbbk[x_1,\ldots ,x_n]$ be the polynomial ring over a field $\Bbbk$ and let $G$ be a finite simple graph with vertex set $V(G)=\{x_1,\dots ,x_n\}$ and edge set $E(G)$. The \emph{cover ideal} of $G$ is a squarefree monomial ideal of $S$ defined by 
$$\displaystyle J(G)=\bigcap_{\{x_i,x_j\}\in E(G)} (x_i, x_j).$$
The cover ideal $J(G)$ is the Alexander dual of the well-known edge ideal of $G$. Cover ideals and their powers were studied in many articles, see for example \cite{constan, erey, second power,  fakhari reg,fakhari symbolic,francisco van tuyl,hang,herzog hibi ohsugi,kumar, mohammadi powers of chordal,mohammadi2} Herzog, Hibi and Ohsugi \cite{herzog hibi ohsugi} showed that if $G$ is a Cohen-Macaulay chordal graph, then all powers of the cover ideal of $G$ have linear resolutions. Moreover, they proposed the following conjecture:

\begin{conjecture}\label{conjecture}\cite[Conjecture~2.5]{herzog hibi ohsugi} All powers of the vertex cover ideal of a chordal graph are componentwise linear.
\end{conjecture}
Francisco and Van Tuyl \cite{francisco van tuyl} showed that cover ideals of chordal graphs are componentwise linear. For a graded ideal $I\subset S$, being componentwise linear is an algebraic property which requires that for all $j$, the ideal $I_{\langle j \rangle}$, generated by all homogeneous polynomials of degree $j$ belonging to $I$, has a linear resolution. Later, it was proved that chordal graphs in fact have stronger combinatorial properties such as being shellable \cite{vv} and vertex decomposable \cite{woodroofe}. In \cite{mohammadi powers of chordal} it was proved that powers of cover ideals of Cohen-Macaulay chordal graphs have linear quotients. A graph $G$ is called Cohen-Macaulay if the quotient ring $S/I(G)$ is Cohen-Macaulay, where $I(G)$ denotes the edge ideal of $G$. It is well-known \cite[Lemma 9.1.10]{herzog hibi monomial ideals} that the cover ideal of a Cohen-Macaulay graph is generated in single degree. Since cover ideal of a path graph can have minimal generators of different degrees, paths are not necessarily Cohen-Macaulay. In fact, using the recursive description of the minimal generating set of $J(P_n)$ in Lemma~\ref{lem: rooted list gives linear quotients}, one can show that $P_n$ is not Cohen-Macaulay for $n\geq 5$.

In a recent preprint \cite{kumar} Kumar and Kumar proved Conjecture~\ref{conjecture} for all trees. Their main tool is a result from \cite{fakhari symbolic} which says that for any graph $G$ the polarization of $k^{th}$ symbolic power of $J(G)$ is the cover ideal of some graph denoted by $G_k$. Since symbolic powers and ordinary powers of cover ideal of bipartite graphs coincide \cite{gitler}, their approach is to show that $G_k$ is vertex decomposable when $G$ is a tree. Although trees contain the class of path graphs, the methods in \cite{kumar} cannot be applied to non-bipartite chordal graphs.

The main goal of this paper is to make a contribution to the problem in Conjecture~\ref{conjecture} and bring up an idea that is applicable to all chordal graphs. We introduce the notion of rooted order (Definition~\ref{def:rooted for high powers} and Definition~\ref{def:rooted for powers of chordal}) and we show that all powers of the cover ideal of a path graph have linear quotients with respect to such order (Theorem~\ref{thm:main thm}). Our results build on and extend the analogous results presented in \cite{second power} from second powers to all powers. We analyze the minimal generating set of $J(P_n)^s$ in relation to rooted order. An interesting byproduct we obtain in the process is Corollary~\ref{cor: characterization} which characterizes the minimal generators of $J(P_n)^s$ for $s\geq 3$ in terms of those of the second power. Although we focus on the class of path graphs, the notion of rooted order naturally generalizes to chordal graphs. In fact, examples we tested on chordal graphs led us to question if one can always find a rooted order which gives linear quotients for powers of their cover ideals. We discuss this in Section~\ref{sec: rooted order on chordal graphs} and we think that the techniques developed in this article may be helpful to further explore the problem at hand in a more general framework.
 
\section{Preliminaries}
Let $S=\Bbbk[x_1,\ldots ,x_n]$ be the polynomial ring over a field $\Bbbk$ and let $I$ be a monomial ideal. We denote the set of minimal generators of $I$ by $G(I)$. We say $I$ has \emph{linear quotients} if there exists an order $u_1,\dots ,u_k$ on the elements of $G(I)$ such that for every $i=2,\dots ,k$ the colon ideal $(u_1,\ldots ,u_{i-1}):(u_i)$ is generated by some variables. To simplify our notation, for any pair of monomials $u$ and $v$ we will write
\[\displaystyle u:v = \frac{u}{\gcd(u,v)}.\]
If $M$ is a subset of $S$ and $u$ is a monomial, then we define a new subset $uM$ by
\[uM = \{um : m\in M\}.\]
Similarly, if $L=v_1,\ldots ,v_t$ is a list (or sequence) of monomials, then $uL$ denotes a new list obtained from $L$ by multiplying each term by $u$. In other words, 
\[uL=uv_1,\ldots ,uv_t.\]
To keep our notation simple and also to distinguish lists from ideals we will not put parentheses around lists.

Let $G$ be a finite simple graph with vertex set $V(G)=\{x_1,\dots ,x_n\}$ and edge set $E(G)$. A set $C$ of vertices of $G$ is called a \emph{vertex cover} if $e\cap C\neq \emptyset$ for every edge $e\in E(G)$. A vertex cover $C$ is called \emph{minimal} if no proper subset of $C$ forms a vertex cover for $G$. The \emph{cover ideal} of $G$ is denoted by $J(G)$ and it is defined by \[J(G)=\bigcap_{\{x_i,x_j\}\in E(G)} (x_i, x_j).\] The set of minimal generators of $J(G)$ is given by
\[G(J(G))=\{x_{i_1}\dots x_{i_k} : \{x_{i_1},\dots, x_{i_k}\} \text{ is a minimal vertex cover of } G\}.\]
If the graph $G$ has no edges, then $J(G)=(1)$.

If $A$ is a subset of vertices of $G$, then $G\setminus A$ denotes the graph which is obtained from $G$ by removing the vertices in $A$. We call a graph \emph{chordal} if it has no induced cycle of length greater than $3$. We say $x_i$ is a \emph{neighbor} of $x_j$ if $\{x_i,x_j\}\in E(G)$. The set of all neighbors of $x_i$ is denoted by $N(x_i)$. The \emph{closed neighborhood} of $x_i$ is denoted by $N[x_i]$ and it is equal to the union $N(x_i)\cup \{x_i\}$. Every chordal graph has a vertex whose closed neighbourhood induces a complete graph and such vertex is called a \emph{simplicial vertex}.

A \emph{path} on vertices $x_1,\dots ,x_n$ is denoted by $P_n$. Throughout the paper we will assume that edges of $P_n$ are labelled as
\[E(G)=\{\{x_1,x_2\},\{x_2,x_3\},\dots ,\{x_{n-1}, x_n\}\}.\]
 
\begin{Definition}[\textbf{Rooted list}]\cite[Definition~2.2 ]{second power}\label{def:rooted list} The {\em rooted list} of $P_n$, denoted by  $\MR(P_n)$, is recursively defined by the following formulas:
	\begin{itemize}
		\item $\MR(P_1)=1$ 
		\item  $\MR(P_2)=x_1, x_2$
		
		\item  $\MR(P_3)=x_2,x_1x_3$
		
		%\item $\MR(P_4)=x_1x_3,x_2x_3, x_2x_4$
		
		\item for $n\geq 4$, if $\MR(P_{n-2})=u_1,\dots , u_r$ and $\MR(P_{n-3})=v_1,\dots , v_s$ then
		\[\MR(P_n)=x_{n-1}u_1,\dots ,x_{n-1}u_r, x_nx_{n-2}v_1,\dots ,x_nx_{n-2}v_s.\]	
	\end{itemize}
	\end{Definition}
The motivation for this definition is the next lemma.

\begin{Lemma}\label{lem: rooted list gives linear quotients}
	Let $\MR(P_n)=u_1,\dots ,u_q$. Then
	\begin{enumerate}
		\item $G(J(P_n))=\{u_1,\dots ,u_q\}$
		\item $J(P_n)$ has linear quotients with respect to $u_1,\dots ,u_q$.
	\end{enumerate}
\end{Lemma}
\begin{proof}
	Follows from \cite[Lemma~2.1]{second power} and the recursive definition of rooted list.
\end{proof}

Based on the lemma above, a total order on minimal generators of $J(P_n)$ was defined.

\begin{Definition}[\textbf{Rooted order}]\cite[Definition~2.2 ]{second power}\label{def:rooted order}
	Let $\MR(P_n)=u_1,\dots ,u_q$. The \emph{rooted order}, denoted by $>_\MR$, is a total order on $G(J(P_n))$ such that $u_i>_\MR u_j$ when $i<j$.
\end{Definition}

\begin{Definition}
	Let $\mathbf{u}=(u_1, \ldots ,u_n)$, $\mathbf{v}=(v_1, \ldots, v_n)$ be two elements in $\ZZ^n$. Then we write $\mathbf{u}>_{\lex} \mathbf{v}$ if the first non-zero entry in $\mathbf{u}-\mathbf{v}$ is positive. 
\end{Definition}

The following is a general version of Definition~2.4 in \cite{second power}.
\begin{Definition}[\textbf{$\mathbf{s}$-fold product, maximal expression}]
	Let $I=(u_1,\dots ,u_q)$.	We say that $M=u_1^{a_1}\dots u_q^{a_q}$ is an {\em $s$-fold product} of minimal generators of $I$ if each $a_i$ is a non-negative integer and $a_1+\dots +a_q=s$. We write $u_1^{a_1}\dots u_q^{a_q}>_{lex} u_1^{b_1}\dots u_q^{b_q}$ if $(a_1,\dots ,a_q)>_{lex} (b_1,\dots ,b_q)$.
	We say that $M=u_1^{a_1}\dots u_q^{a_q}$ is the {\em maximal expression} if $(a_1,\dots ,a_q)>_{lex} (b_1,\dots ,b_q)$ for any other $s$-fold product $M=u_1^{b_1}\dots u_q^{b_q}$. 
\end{Definition}

\begin{Notation}
	If $G(I)=\{u_1,\dots , u_q\}$, then the set of all $s$-fold products is denoted by $F(I^s)=\{u_{i_1}\dots u_{i_s} :  u_{i_1},\dots, u_{i_s}\in G(I)\}$.
\end{Notation}

We also generalize Definition~2.6 in \cite{second power} from second powers to all powers.
\begin{Definition}[\textbf{Rooted order/list on powers}]\label{def:rooted for high powers} Let $\MR (P_n)=u_1,\dots ,u_q$. We define a total order $>_{\MR}$ on $F(J(P_n)^s)$ which we call \emph{rooted order} as follows. For $M,N\in F(J(P_n)^s)$ with maximal expressions $M=u_1^{a_1}\dots u_q^{a_q}$ and $N=u_1^{b_1}\dots u_q^{b_q}$ we set $M>_{\MR} N$ if $(a_1,\dots ,a_q)>_{lex} (b_1,\dots ,b_q)$. 
	
	Let $G(J(P_n)^s)=\{U_1, \ldots, U_r\}$. Then we say $U_1,\dots ,U_r$ is a \emph{rooted list} of minimal generators of $J(P_n)^s$ if $U_1>_{\MR} \ldots >_{\MR} U_r$. In such case, we denote the rooted list of generators by $\MR(J(P_n)^s)=U_1, \ldots, U_r.$ 
\end{Definition}

\begin{Remark}
	If $n=1$, then $\MR(J(P_n)^s)=1$ for every $s$.
\end{Remark}

%---------------------------------------
%    Rooting of P_n
%---------------------------------------
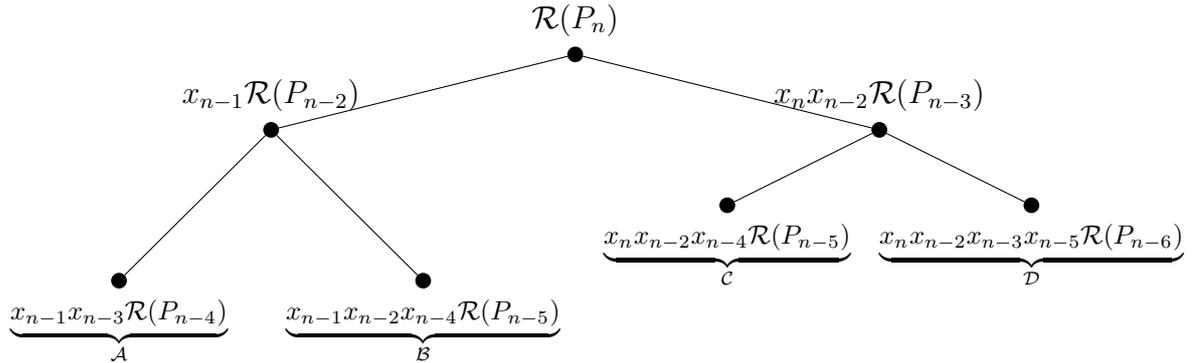
\begin{figure}[hbt!]
	\centering
	\begin{tikzpicture}
	[scale=1.00, vertices/.style={draw, fill=black, circle, minimum size = 4pt, inner sep=0.5pt}, another/.style={draw, fill=black, circle, minimum size = 5.5pt, inner sep=0.1pt}]
	\node[another, label=above:{$\MR(P_{n})$}] (a) at (-1,0) {};
	\node[another, label=above:{$x_{n-1}\MR(P_{n-2})$}] (b) at (-5, -1) {};
	\node[another, label=above:{$x_{n}x_{n-2}\MR(P_{n-3})$}] (c) at (3, -1) {};
	\node[another, label=below:{\footnotesize$\underbrace{x_{n-1}x_{n-2}x_{n-4}\MR(P_{n-5})}_\mathcal{B}$}] (d) at (-3, -3) {};
	\node[another, label=below:{\footnotesize$\underbrace{x_{n-1}x_{n-3}\MR(P_{n-4})}_\mathcal{A}$}] (e) at (-7, -3) {};
	\node[another, label=below:{\footnotesize$\underbrace{x_{n}x_{n-2}x_{n-4}\MR(P_{n-5})}_\mathcal{C}$}] (f) at (1, -2) {};
	\node[another, label=below:{\footnotesize$\underbrace{x_{n}x_{n-2}x_{n-3}x_{n-5}\MR(P_{n-6})}_\mathcal{D}$}] (g) at (5, -2) {};
	\foreach \to/\from in {a/b, a/c, b/d, b/e, c/f, c/g}
	\draw [-] (\to)--(\from);
	\end{tikzpicture}
	\caption{\label{fig:P_n} $\MR(P_{n})$ in $2$ steps}
\end{figure}

%---------------------------------------------------
%      Convention
%---------------------------------------------------

\begin{Remark}\label{rk:convention}
	If $n=6$, then Figure~\ref{fig:P_n} is still valid if we make the convention $\MR(P_0)=1$. In this case, the lists $\mathcal{B, C, D}$ each has only one term:
	\[\mathcal{B}= x_{n-1}x_{n-2}x_{n-4}, \, \mathcal{C}= x_{n}x_{n-2}x_{n-4}, \, \mathcal{D}= x_{n}x_{n-2}x_{n-3}x_{n-5}.\]
\end{Remark}

%---------------------------------------------------------
% Minimal generators of powers of J(P_n)
%--------------------------------------------------------

\section{Properties of rooted order and $G(J(P_n)^s)$}
In this section, we will establish some properties of rooted order and minimal generating set of $J(P_n)^s$ which will be useful in the sequel.
\begin{Remark}\label{rk:divisor of the same type}
	Observe that if $n\geq 2$, then every minimal vertex cover of $P_n$ contains either $x_n$ or $x_{n-1}$, but not both. Therefore if $U,V\in F(J(P_n)^s)$ such that $U|V$, then the highest power of $x_n$ (respectively $x_{n-1}$) dividing $U$ is the same as that of $x_n$ (respectively $x_{n-1}$) dividing $V$.
\end{Remark}

\begin{Lemma}\label{lem:every b contains some a}\cite[Lemma~3.5]{second power} 
	Let $n\geq 3$ and let $u \in G(J(P_n)) $ such that $x_n | u$. Then there exists $v \in  G(J(P_{n-2}))$ such that $v$ divides $u/x_n$.
\end{Lemma}

%--------------------------------------
% PRODUCT OF ALL A'S OR B'S
%------------------------------------

\begin{Remark}\label{rk:max expression preserved P_{n-2}}
	Let $n\geq 4$ and let $\MR(P_{n-2})=u_1,\dots ,u_m$. Observe that by definition of rooted order the expression $u_{i_1}\ldots u_{i_s}$ is maximal with $i_1\leq \dots \leq i_s$ if and only if $(x_{n-1}u_{i_1})\dots (x_{n-1}u_{i_s})$ is the maximal expression with $x_{n-1}u_{i_1} \geq_{\MR} \dots \geq_{\MR} x_{n-1}u_{i_s}$ in $\MR(P_n)$.
\end{Remark}

\begin{Remark}\label{rk:max expression preserved P_{n-3}}
	Let $n\geq 5$ and let $\MR(P_{n-3})=u_1,\dots ,u_m$. Observe that by definition of rooted order the expression $u_{i_1}\dots u_{i_s}$ is maximal with $i_1\leq \dots \leq i_s$ if and only if $(x_nx_{n-2}u_{i_1})\dots (x_nx_{n-2}u_{i_s})$ is the maximal expression with $x_nx_{n-2}u_{i_1} \geq_{\MR} \dots \geq_{\MR} x_nx_{n-2}u_{i_s}$ in $\MR(P_n)$.		
\end{Remark}

According to recursive definition of rooted list, for $n\geq 4$ each factor of an $s$-fold product of minimal generators of $J(P_n)$ belongs to either $x_{n-1}\MR(P_{n-2})$ or $x_nx_{n-2}\MR(P_{n-3})$. If all of the factors are from $x_{n-1}\MR(P_{n-2})$ or all of the factors are from $x_nx_{n-2}\MR(P_{n-3})$, then the $s$-fold product is \emph{pure}. Otherwise $s$-fold product is \emph{mixed}. Now, we make some observations on pure and mixed $s$-fold products.
\begin{Lemma}[\textbf{Pure} $\mathbf{s}$\textbf{-fold product divisible by} $\mathbf{x_{n-1}^s}$]\label{lem:product of all As}
	Let $n\geq 3$. Then $U,V\in F(J(P_{n-2})^s)$ if and only if $x_{n-1}^sU, x_{n-1}^sV\in F(J(P_n)^s)$. Moreover, in such case the following statements hold.
	\begin{enumerate}
		\item  $U>_{\MR} V$ if and only if $x_{n-1}^sU>_{\MR} x_{n-1}^sV$.
		\item $U\in G(J(P_{n-2})^s)$ if and only if $x_{n-1}^sU\in G(J(P_n)^s)$.
	\end{enumerate}
	
\end{Lemma}
\begin{proof}
	The fist statement is clear from the definition of rooted list and Lemma~\ref{lem: rooted list gives linear quotients}. To see (1) let $\MR(P_{n-2})=u_1,\dots ,u_m$.	Suppose that $U=u_{i_1}\dots u_{i_s}$ and $V=u_{j_1}\dots u_{j_s}$ are maximal expressions with $i_1\leq \dots \leq i_s$ and $j_1\leq \dots \leq j_s$. Then by Remark~\ref{rk:max expression preserved P_{n-2}} the expressions $(x_{n-1}u_{i_1})\dots (x_{n-1}u_{i_s})$ and $(x_{n-1}u_{j_1})\dots (x_{n-1}u_{j_s})$ are maximal as well. Suppose $U\neq V$ and let $t$ be the smallest index such that $i_t\neq j_t$. Then
	\[U>_\MR V \Longleftrightarrow \, i_t<j_t \, \Longleftrightarrow \, x_{n-1}^sU>_{\MR} x_{n-1}^sV\]
	as desired. For proof of (2), the direction ($\Leftarrow$) is straightforward and the direction ($\Rightarrow$) follows from Remark~\ref{rk:divisor of the same type}.
\end{proof}

%-------------------------
%    ALL B's
%-------------------------
\begin{Lemma}[\textbf{Pure} $\mathbf{s}$\textbf{-fold product divisible by} $\mathbf{x_{n}^s}$]\label{lem:product of all Bs}
	Let $n\geq 4$. Then $U,V\in F(J(P_{n-3})^s)$ if and only if both $x_n^sx_{n-2}^sU$ and $x_n^sx_{n-2}^sV$ belong to $F(J(P_n)^s)$. Moreover, in such case the following statements hold.
	\begin{enumerate}
		\item $U>_{\MR} V$ if and only if $x_n^sx_{n-2}^sU>_{\MR}x_n^s x_{n-2}^sV$.
		\item $U\in G(J(P_{n-3})^s)$ if and only if $x_n^sx_{n-2}^sU\in G(J(P_n)^s)$
	\end{enumerate}	 
\end{Lemma}
\begin{proof}
	Similar to proof of Lemma~\ref{lem:product of all As} using Remark~\ref{rk:max expression preserved P_{n-3}}.
\end{proof}
%-----------------------------
%    MIXED PRODUCT
%----------------------------
\begin{Lemma}[\textbf{Mixed} $\mathbf{s}$\textbf{-fold product}]\label{lem:mixed product maximal expression and minimal generator}
	Let $\MR(P_{n-2})=u_1,\dots ,u_a$ and let $\MR(P_{n-3})=v_1,\dots ,v_b$ for some $n\geq 4$. Let $U=u_{i_1}\dots u_{i_q}$, $V=v_{j_1}\dots v_{j_k}$ and $W=x_{n-1}^qx_n^kx_{n-2}^kUV$.
	\begin{enumerate}
		\item If $W=(x_{n-1}u_{i_1})\dots (x_{n-1}u_{i_q})(x_nx_{n-2}v_{j_1})\dots (x_nx_{n-2}v_{j_k})$ is the maximal expression in $F(J(P_n)^{k+q})$ with $x_{n-1}u_{i_1} \geq_\MR \dots \geq_\MR x_{n-1}u_{i_q} >_\MR x_nx_{n-2}v_{j_1} \geq_\MR\dots \geq_\MR x_nx_{n-2}v_{j_k}$, then the expression $U=u_{i_1}\dots u_{i_q}$ is maximal in $F(J(P_{n-2})^q)$ with $i_1\leq \dots \leq i_q$ and the expression $V=v_{j_1}\dots v_{j_k}$ is maximal in $F(J(P_{n-3})^k)$  with $j_1\leq \dots \leq j_k$.
\item If $W\in G(J(P_n)^{q+k})$, then $U\in G(J(P_{n-2})^q)$ and $V\in G(J(P_{n-3})^k)$.
\end{enumerate}
\end{Lemma}
	\begin{proof}
	Proof is straightforward and left to the reader.
	\end{proof}

Note that in the previous lemma, the converses of (1) and (2) are not true. 
\begin{itemize}
	\item Consider $q=k=1$ and $n=7$ with $\MR(P_7)=u_1,\dots ,u_7$. Then $u_4\in x_6\MR(P_5)$ and $u_5\in x_7x_5\MR(P_4)$ but $u_4u_5$ is not the maximal expression as $u_3u_6=u_4u_5$. Thus the converse of (1) is not true.
	\item Consider $q=k=1$ and $n=5$. Then $U=x_1x_3\in G(J(P_3))$ and $V=x_2\in G(J(P_2))$ but $(x_4U)(x_3x_5V)\notin G(J(P_5)^2)$. Thus the converse of (2) is not true.
\end{itemize}

%--------------------------
%  SMALL CASES
%---------------------------
\subsection{Reduction to second powers}
In this section we will reduce the problem of describing minimal generating set of $J(P_n)^s$ to the case when $s=2$. To this end, first we will explicitly describe $G(J(P_n)^s)$ for some small values of $n$. These results will then form the basis step of inductive proof of Theorem~\ref{thm:reduction} which will be our next goal.
\begin{Lemma}\label{lem: s-fold products of paths at most 4 vertices} If $2\leq n\leq 4$, then $G(J(P_n)^s)=F(J(P_n)^s)$ for all $s$. Moreover, in that case every $U\in F(J(P_n)^s)$ has a unique expression as an $s$-fold product of minimal generators of $J(P_n)$. 
\end{Lemma}
\begin{proof}
\emph{Case 1:} Suppose $n=2$ or $n=3$. Then $\MR(P_n)=u_1,u_2$ where $x_{n-1}$ divides $u_1$ and $x_n$ divides $u_2$. Let $V=u_1^\alpha u_2^\beta$ be an $s$-fold product which divisible by another $s$-fold product $U=u_1^au_2^b$. Since the exponents of $x_n$ in $U$ and $V$ are respectively $b$ and $\beta$ it follows from Remark~\ref{rk:divisor of the same type} that $b=\beta$. Similarly, since the exponents of $x_{n-1}$ are equal we get $a=\alpha$ and $U=V$.

\emph{Case 2:} Suppose $n=4$. Then $\MR (P_4)=u_1, u_2, u_3$ where $u_1=x_1x_3, u_2=x_2x_3, u_3=x_2x_4$. Let $U=u_1^au_2^bu_3^c$ and $V=u_1^\alpha u_2^\beta u_3^\gamma $ be $s$-fold products such that $U$ divides $V$. Remark~\ref{rk:divisor of the same type} implies that $c=\gamma$ and $a+b=\alpha+\beta$. Since the exponents of $x_1$ in $U$ and $V$ are respectively $a$ and $\alpha$ it follows that $a\leq \alpha$. Similarly, comparing exponents of $x_2$ we get $b\leq \beta$. Thus $a=\alpha$, $b=\beta$ and $U=V$.
\end{proof}

\begin{Lemma}\label{lem: s-fold products of paths 5 vertices}
	Let $\MR(P_5)=u_1,u_2,u_3,u_4$. If $U=u_1^\alpha u_2^\beta u_3^\gamma u_4^\delta \in F(J(P_5)^s)\setminus G(J(P_5)^s)$, then $\beta, \delta>0$.
\end{Lemma}
\begin{proof}
	Let $u_1=x_2x_4$, $u_2=x_1x_3x_4$, $u_3=x_1x_3x_5$, $u_4=x_2x_3x_5$. Let $V=u_1^au_2^bu_3^cu_4^d\in G(J(P_5)^s)$ such that $V|U$. First note that by Remark~\ref{rk:divisor of the same type} we have
	\begin{equation}\label{eq:both s-fold}a+b=\alpha+\beta \, \text{ and } \, c+d=\gamma+\delta .\end{equation}
 Moreover, since the degree of $V$ is less than degree of $U$ we have
	\begin{equation}\label{eq:degrees equal}
	2a+3b+3c+3d < 2\alpha+3\beta +3\gamma +3\delta.
	\end{equation}
	Combining \eqref{eq:both s-fold} and \eqref{eq:degrees equal} we obtain $b<\beta$ and thus $\beta >0$. Then we get $a>\alpha$. Comparing the exponents of $x_2$ in $U$ and $V$ we get $a+d\leq \alpha+\delta$ and thus $\delta>0$.
\end{proof}

\begin{Lemma}\label{lem:trivial side of reduction characterization thm}
	Let $\MR(P_n)=u_1,\dots ,u_r$ with $n\geq 2$ and let $s\geq 2$.
	\begin{enumerate}
		\item If $u_{i_1}\dots u_{i_s}\in G(J(P_n)^s)$, then $u_pu_q\in G(J(P_n)^2)$ for all $p,q\in\{i_1,\dots ,i_s\}$.
		
		\item If $u_{i_1}\dots u_{i_s}$ is the maximal expression for some $i_1\leq \dots \leq i_s$, then for all $p,q\in\{i_1,\dots ,i_s\}$ with $p<q$ the expression $u_pu_q$ is maximal.
		
	\end{enumerate}
\end{Lemma}
\begin{proof}
	To see (1) assume for a contradiction $u_{i_1}\dots u_{i_s}\in G(J(P_n)^s)$ but there exist $p,q\in\{i_1,\dots ,i_s\}$ such that $u_pu_q\notin G(J(P_n)^2)$. Then there exists $u_{p'}u_{q'}\in G(J(P_n)^2)$ which strictly divides $u_pu_q$. Then $u_{i_1}\dots u_{i_s}u_{p'}u_{q'}/(u_pu_q)$ is an $s$-fold product and it strictly divides $u_{i_1}\dots u_{i_s}$, contradicting our initial assumption.
	Proof of (2) is similar.
\end{proof}
%---------------------------------------------------
%            REDUCTION THEOREM
%---------------------------------------------------

\begin{Theorem}\label{thm:reduction}
	Let $G(J(P_n))=\{u_1,\dots ,u_r\}$ with $n\geq 2$ and $s\geq 2$. Let $U=u_1^{a_1}\dots u_r^{a_r}$ be an $s$-fold product in $F(J(P_n)^s)$. If $U\notin G(J(P_n)^s)$, then there exist $p$ and $q$ with $a_p,a_q>0$ such that $u_pu_q \notin G(J(P_n)^2)$.
\end{Theorem}
\begin{proof}
	We use induction on $n$. Suppose that $U\notin G(J(P_n)^s)$. If $n\leq 4$, then the statement is vacuously true by Lemma~\ref{lem: s-fold products of paths at most 4 vertices}. If $n=5$, then $u_1u_3$ strictly divides $u_2u_4$ and the statement is true by Lemma~\ref{lem: s-fold products of paths 5 vertices}. Therefore let us assume that $n\geq 6$.
	
	 Keeping Figure~\ref{fig:P_n} in mind, observe that if $x_{n-1}^s$ divides $U$, then the result follows from Lemma~\ref{lem:product of all As} and the induction assumption on $P_{n-2}$. Similarly, if $x_{n}^s$ divides $U$, then the result follows from Lemma~\ref{lem:product of all Bs} and the induction assumption on $P_{n-3}$. Therefore, let us assume that $U$ is divisible by $x_nx_{n-1}$.

	If there exist $p$ and $q$ with $a_p,a_q>0$ such that $x_{n-4}x_{n-1}|u_p$ and $x_{n-3}x_n|u_q$, then the result follows from \cite[Lemma~4.1]{second power}.  Therefore, it suffices to consider the following cases:
	
	\emph{Case 1:} Suppose that $U$ is product of factors from $\mathcal{A, B,C}$ in Figure~\ref{fig:P_n} such that at least one factor from $\mathcal{A}$ or $\mathcal{B}$ is divisible by $x_{n-4}$. Then we can write
	\[U=(x_{n-1}x_{n-3})^\alpha V (x_{n-1}x_{n-2}x_{n-4})^\beta W(x_{n}x_{n-2}x_{n-4})^\gamma Y\]
	for some $V\in F(J(P_{n-4})^\alpha)$, $W\in F(J(P_{n-5})^\beta)$, $Y\in F(J(P_{n-5})^\gamma)$.
	Let $U'\in G(J(P_n)^s)$ such that $U'$ strictly divides $U$. Keeping Remark~\ref{rk:convention} in mind, suppose that 
	\[U'=(x_{n-1}x_{n-3})^{\alpha'} V' (x_{n-1}x_{n-2}x_{n-4})^{\beta'} W'(x_{n}x_{n-2}x_{n-4})^{\gamma'} Y' (x_{n}x_{n-2}x_{n-3}x_{n-5})^{\delta'} Z'\]
	for some $V'\in F(J(P_{n-4})^{\alpha'})$, $W'\in F(J(P_{n-5})^{\beta'})$, $Y'\in F(J(P_{n-5})^{\gamma'})$, $Z'\in F(J(P_{n-6})^{\delta'})$.
	We claim that 
	\[(\alpha, \beta, \gamma, 0)= (\alpha', \beta', \gamma', \delta').\]
	By Remark~\ref{rk:divisor of the same type} we have $\alpha+\beta=\alpha'+\beta'$ and $\gamma=\gamma'+\delta'$. Since the exponent of $x_{n-2}$ in $U'$ is less than or equal to that of $U$ we have
	\[\beta+\gamma\geq \beta'+\gamma'+\delta'.\]
	Similarly, since the exponent of $x_{n-3}$ in $U'$ is less than or equal to that of $U$ we have
	\[\alpha \geq \alpha'+\delta'.\]
	Then adding up the inequalities we get $\delta'=0$. Then $\gamma=\gamma'+\delta'$ implies $\gamma=\gamma'$. Therefore $\alpha=\alpha'$ and $\beta=\beta'$ as desired.
	
	Therefore, $V'W'Y'$ strictly divides $VWY$. By recursive definition of $\MR(P_{n-2})$  (see Figure~\ref{fig:P_n-2}) observe that 
	$$U^*=x_{n-3}^\alpha V (x_{n-2}x_{n-4})^\beta W (x_{n-2}x_{n-4})^\gamma Y \in F(J(P_{n-2})^s)\setminus G(J(P_{n-2})^s).$$
	Then by induction assumption on $P_{n-2}$, one of $V,W$ or $Y$ contains a non-minimal $2$-fold product. By adding the suitable variables, one can see that $U$ satisfies the desired condition.
	%----------------------------------------------
	%     CASE 2
	%----------------------------------------------
	
	\emph{Case 2:} Suppose that $U$ is product of factors from $\mathcal{A,C,D}$ such that no factor from $\mathcal{A}$ is divisible by $x_{n-4}.$ Then we can write
		\[U=(x_{n-1}x_{n-3})^\mu V (x_{n}x_{n-2})^\nu X\] for some $V\in F(J(P_{n-4})^\mu)$, $X\in F(J(P_{n-3})^\nu)$, where $\mu, \nu >0$ and $\mu+\nu=s$. We claim that $U$ is divisible by some $U'\in G(J(P_n)^s)$ of the same form. Indeed, if 
		\[U'=(x_{n-1}x_{n-3})^{\mu'} V' (x_{n-1}x_{n-2}x_{n-4})^{\beta'}W'(x_{n}x_{n-2})^{\nu'} X'\]
		for some $V'\in F(J(P_{n-4})^{\mu'})$, $W'\in F(J(P_{n-5})^{\beta'})$ and $X' \in F(J(P_{n-3})^{\nu'})$, then we must have $\mu'+\beta'=\mu$ and $\nu=\nu'$ by Remark~\ref{rk:divisor of the same type}. Then comparing the exponents of $x_{n-2}$ in $U$ and $U'$ we see that $\beta'=0$.
		
		Therefore, $V'X'$ strictly divides $VX$. Then by recursive definition of $\MR(P_{n-1})$ observe that 
		$$U^*=x_{n-2}^\nu X (x_{n-1}x_{n-3})^\mu V \in F(J(P_{n-1})^s)\setminus G(J(P_{n-1})^s). $$
		Then by induction assumption on $P_{n-1}$, either $V$ or $X$ contains a non-minimal $2$-fold product. By adding the suitable variables, one can see that $U$ satisfies the desired condition.  
\end{proof}
%-----------------------------------------
%        Rooting of P_{n-2}
%-----------------------------------------
\begin{figure}[hbt!]
	\centering
	\begin{tikzpicture}
	[scale=1.00, vertices/.style={draw, fill=black, circle, minimum size = 4pt, inner sep=0.5pt}, another/.style={draw, fill=black, circle, minimum size = 5.5pt, inner sep=0.1pt}]
	\node[another, label=above:{$\MR(P_{n-2})$}] (a) at (0,0) {};
	\node[another, label=below:{$x_{n-3}\MR(P_{n-4})$}] (b) at (-2, -1) {};
	\node[another, label=below:{$x_{n-2}x_{n-4}\MR(P_{n-5})$}] (c) at (2, -1) {};
	
	\foreach \to/\from in {a/b, a/c}
	\draw [-] (\to)--(\from);
	\end{tikzpicture}
	\caption{\label{fig:P_n-2} Recursive definition of rooted list of $P_{n-2}$}
\end{figure}
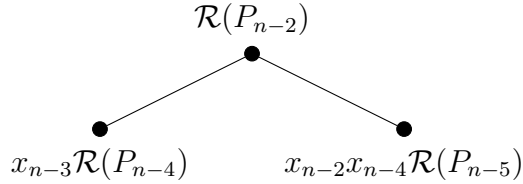
As a consequence of Theorem~\ref{thm:reduction} we characterize minimal generating set of $J(P_n)^s$ for $s\geq 3$ in terms of minimal generating set of second power of $J(P_n)$.

%------------------------------------------
%  CHARACTERIZATION
%-----------------------------------------
\begin{Corollary}\label{cor: characterization}
	Let $G(J(P_n))=\{u_1, \dots ,u_r\}$ and let $s\geq 2$. The following statements are equivalent.
	\begin{enumerate}
		\item $u_{i_1}\dots u_{i_s}\in G(J(P_n)^s)$.
		\item $u_pu_q\in G(J(P_n)^2)$ for all $p,q\in\{i_1,\dots ,i_s\}$.
		\end{enumerate}
\end{Corollary}
\begin{proof}
	Immediate from Lemma~\ref{lem:trivial side of reduction characterization thm} and Theorem~\ref{thm:reduction}.
\end{proof}

Given a monomial ideal $I$, let $\mu(I)$ denote the cardinality of $G(I)$. If $G(I)=\{u_1,\dots ,u_q\}$, then by counting the number of $s$-element multi-subsets of $[q]=\{1,\dots ,q\}$ one can see that $\mu(I^s)\leq {{q+s-1}\choose{q-1}}$. This upper bound may not be achieved in general for two reasons. Firstly, a product of the form $u_{i_1}\dots u_{i_s}$ may be equal to another product $u_{j_1}\dots u_{j_s}$ with $\{i_1,\dots ,i_s\}\neq \{j_1,\dots ,j_s\}$ as multi-sets. Secondly, $u_{i_1}\dots u_{i_s}$ may be strictly divisible by another product $u_{j_1}\dots u_{j_s}$. In fact, when $I$ is generated by monomials of the same degrees, the latter cannot happen. Therefore, although the computation of $\mu(I^s)$ is a challenging problem, one can describe the set $G(I^s)$ explicitly when $I$ is generated in the same degree. On the other hand, when $I$ is not generated in the same degree, description of $G(I^s)$ remains a difficult problem as well as computation of $\mu(I^s)$.

It is well-known (\cite{kodiyalam}) that the function $g(s)=\mu(I^s)$ is a polynomial in $s$ for $s\gg 0$. In \cite{h1}, the authors addressed the question of how small $\mu(I^2)$ can be in terms of $\mu(I)$ when $I$ is a monomial ideal in polynomial ring with $n=2$ variables. Behaviour of $\mu(I^s)$ was considered in some other articles, see for example \cite{abdolmaleki, gasanova, h2, freiman}. Recently, Drabkin and Guerrieri \cite{drabkin} studied Freiman cover ideals. Given a cover ideal $J(G)$, it is a demanding task to find the minimal generating set of $J(G)^s$ or $\mu(J(G)^s)$. Therefore, Corollary~\ref{cor: characterization} might be of interest in computation of $\mu(J(P_n)^s)$.

We will next see how Theorem~\ref{thm:reduction} will be useful to extend the following result to all powers of $J(P_n)$.

\begin{Lemma}\label{lem:second power crucial}\cite[Lemma~4.5]{second power}
		Let $U \in F(J(P_n)^2)\setminus G(J(P_n)^2)$. Then there exists  $V\in G(J(P_n)^2)$ such that $V>_{\MR}  U$ and $V|U$.
\end{Lemma}

\begin{Lemma}\label{lem:crucial}
	Let $U \in F(J(P_n)^s)\setminus G(J(P_n)^s)$. Then there exists  $V\in G(J(P_n)^s)$ such that $V>_{\MR}  U$ and $V|U$.
\end{Lemma}
\begin{proof}
	Let $\MR(P_n)=u_1, \dots, u_p$ with $n\geq 2$. Let $U=u_1^{\alpha_1}\ldots u_p^{\alpha_p}$ be the maximal expression. By Theorem~\ref{thm:reduction} there exists $u_iu_j$ with $\alpha_i,\alpha_j\neq 0$ and $u_iu_j\notin G(J(P_n)^2)$. Without loss of generality assume that $i\leq j$. Note that $u_iu_j$ is the maximal expression by Lemma~\ref{lem:trivial side of reduction characterization thm}. Then by Lemma~\ref{lem:second power crucial} there exists $v \in G(J(P_n)^2)$ such that $v$ strictly divides $u_iu_j$ and $v>_\MR u_iu_j$. Let $v=u_ku_\ell$ be the maximal expression with $k\leq \ell$. Consider the $s$-fold product $V=(Uu_ku_\ell)/(u_iu_j)$. Observe that $V>_\MR U$ and $V$ strictly divides $U$. If $V$ is a minimal generator, then we are done, otherwise this process can be repeated.
\end{proof}

%------------------------------------------------------------------
%                                MAIN THEOREM
%----------------------------------------------------------------------

\section{Linear quotients of $J(P_n)^s$ with respect to rooted order}
In this section we will show that $J(P_n)^s$ has linear quotients with respect to rooted order. Before that, we prove the following result which will be crucial in the last case of proof of Theorem~\ref{thm:main thm}.
\begin{Proposition}\label{prop:smart claim generalized}
	Let $\MR(P_n)=u_1, \dots , u_q$ and let $\MR(J(P_n)^s)=Y_1, \dots , Y_p$. Suppose that $Y_r=u_{i_1}\dots u_{i_s}$ is the maximal expression for some $2\leq r\leq p$ with $i_1\leq \dots \leq i_s$. 
	\begin{enumerate}
		\item For each $1\leq t \leq s$ with $2\leq i_t$ we have
		\[(u_1,u_2,\dots ,u_{i_t-1}):(u_{i_t})\subseteq (Y_1,\dots ,Y_{r-1}):(Y_r).\]
		\item If $x_n|Y_r$, then $x_{n-1}\in (Y_1,\dots ,Y_{r-1}):(Y_r).$
		\end{enumerate}
\end{Proposition}
\begin{proof}
	(1): First note that by Lemma~\ref{lem: rooted list gives linear quotients} the ideal $(u_1,u_2,\dots ,u_{i_t-1}):(u_{i_t})$ is generated by variables. Let $\ell <i_t$ with $u_\ell:u_{i_t}=x_z$ for some variable $x_z$. Consider the $s$-fold product $M=Y_ru_\ell /u_{i_t}$. Then $M:Y_r=x_z$ and $M\root Y_r$. If $M$ is a minimal generator, nothing is left to show. Otherwise, by Lemma~\ref{lem:crucial} there exists $M'\in G(J(P_n)^s)$ such that $M'>_\MR M$ and $M'|M$. Then since $M'\neq Y_r$ and $M':Y_r$ divides $M:Y_r$ it follows that $M':Y_r=x_z$ and $x_z\in  (Y_1,\dots ,Y_{r-1}):(Y_r)$.
	
	(2): Suppose that $x_n$ divides $u_{i_k}$ for some $k\in\{1,\dots ,s\}$. Then by definition of rooted list $i_k\geq 2$. By part (1) it suffices to show that 
	\[x_{n-1}\in (u_1,u_2,\dots ,u_{i_k-1}):(u_{i_k})\]
	which is immediate from \cite[Lemma~3.6]{second power}.
\end{proof}

We will also need the following result from \cite{second power}.

\begin{Proposition}\label{prop:condition for non-minimal gen or non-maximal expression}\cite[Proposition 4.2]{second power}
	Let $\MR(P_{n})=u_1,\dots , u_k$ where $n\geq 2$.
	Let $1< i<j\leq k$. Suppose $u_j$ contains a variable from $(u_1,\dots ,u_{i-1}):(u_i)$. Then either $u_iu_j$ is not a minimal generator of $J(P_n)^2$ or $u_iu_j$ is not the maximal $2$-fold expression.
\end{Proposition}

We can now prove the main result of this section.

\begin{Theorem}\label{thm:main thm}
	Let $\MR(J(P_n)^s)=Y_1, \ldots, Y_p$. Then $J(P_n)^s$ has linear quotients with respect to $Y_1, \ldots , Y_p$.
\end{Theorem}
\begin{proof} 
	We will proceed by induction on $n+s$. We will show that  $(Y_1,\dots ,Y_{r-1}):(Y_r)$ is generated by variables, for all $r\geq 2$.
	
	\textbf{Basis step ($\mathbf{n\leq 3}$ or $\mathbf{s=1}$):}
	The case when $s=1$ is Lemma~\ref{lem: rooted list gives linear quotients}. If $n=2$ or $n=3$ with $\MR(P_n)=u_1, u_2$, then by Lemma~\ref{lem: s-fold products of paths at most 4 vertices} we have $\MR(J(P_n)^s)=u_1^s,u_1^{s-1}u_2,\dots ,u_2^s$ and it is straightforward to show that
	$(Y_1,\dots ,Y_{r-1}):(Y_r)=(u_1):(u_2)=(x_{n-1})$ holds for every $r\geq 2$.
	
	\textbf{Induction step:} Let us assume that $n\geq 4$ and $s\geq 2$. We set some notation for the following rooted lists.
	
	 \begin{itemize}
	 	\item  $\MR(P_{n-2})=u_1, \dots , u_a$
	 	\item $\MR(P_{n-3})=v_1, \dots , v_b$
	 	\item $\MR(J(P_{n-2})^s)=U_1, \dots , U_A$
	 	\item $\MR(J(P_{n-3})^s)=V_1, \dots ,V_B$.
	 \end{itemize}
 
%-------------------------------------------------------
%              CASE 1
%-------------------------------------------------------
	
\emph{Case 1:} Suppose that $x_n^s$ divides $Y_r$. Assume that $Y_r$ has the maximal expression 
$$Y_r=(x_n x_{n-2}v_{i_1} )(x_n x_{n-2} v_{i_2})\dots (x_n x_{n-2} v_{i_s})$$
for some $i_1\leq \dots \leq i_s$. 
From Proposition~\ref{prop:smart claim generalized} (2) we know that $x_{n-1}$ is a generator of $(Y_1,\dots ,Y_{r-1}):(Y_r)$. From Lemma~\ref{lem:product of all Bs} we can set $V_t=v_{i_1}v_{i_2}\dots v_{i_s}$ for some $t\in\{1,\dots, B\}$. If $t=1$, then by definition of rooted order we have
\[(Y_1,\dots ,Y_{r-1}):(Y_r) =(x_{n-1})\]
and nothing is left to show. Therefore let us assume that $t>1$. Observe that because of induction assumption on $P_{n-3}$ it suffices to show the equality 
\[
 (x_{n-1})+(V_1,V_2,\dots, V_{t-1}):(V_t)=(Y_1,\dots ,Y_{r-1}):(Y_r).
\]
Because of Lemma~\ref{lem:product of all Bs} we already have the inclusion
\[
(x_{n-1})+(V_1,V_2,\dots, V_{t-1}):(V_t) \subseteq (Y_1,\dots ,Y_{r-1}):(Y_r).
\]
We now prove the reverse containment. For any $\ell \leq r-1$, if $x_{n-1} | Y_{\ell}$, then it is clear that $Y_{\ell}:Y_r \in (x_{n-1})$. Otherwise, $(x_{n-2}x_n)^s | Y_{\ell}$ and by Lemma~\ref{lem:product of all Bs}, we have $Y_{\ell}/(x_{n-2}x_n)^s=V_k$ for some $k$. Moreover, since $Y_{\ell} >_{\MR} Y_{r}$ Lemma~\ref{lem:product of all Bs} implies that
$V_k >_{\MR} V_t$. Hence $Y_{\ell}:Y_r=V_k:V_t$, proving the reverse containment.

\medskip

%----------------------------------------------
%   CASE 2
%-----------------------------------------------

\emph{Case 2:} Suppose that $x_{n-1}^s$ divides $Y_r$. Let 
$$Y_r=(x_{n-1}u_{i_1})(x_{n-1}u_{i_2})\dots (x_{n-1}u_{i_s})$$ be the maximal expression for some $i_1\leq \dots \leq i_s$. Then the expression $u_{i_1}\dots u_{i_s}$ is also maximal by Remark~\ref{rk:max expression preserved P_{n-2}}. By Lemma~\ref{lem:product of all As} we can set $U_t=u_{i_1}\dots u_{i_s}$ for some $t>1$ as $r>1$. By induction assumption on $P_{n-2}$ it suffices to show that 
\[
(U_1,\dots ,U_{t-1}):(U_t)=(Y_1,\dots ,Y_{r-1}):(Y_r).
\]

By Lemma~\ref{lem:product of all As} the inclusion $\subseteq$ is clear. To see the reverse, let $\ell\leq r-1$. By definition of rooted list, $Y_{\ell}$ is divisible by either $x_{n-1}^s$ or $x_{n-1}x_n$. Because of Lemma~\ref{lem:product of all As} we may assume that $Y_{\ell}$ is divisible by $x_{n-1}x_n$. Let $Y_{\ell}$ have the maximal expression
\[Y_{\ell}=(x_{n-1}u_{j_1})\dots (x_{n-1}u_{j_c})(x_nx_{n-2}v_{k_{1}})\dots (x_nx_{n-2}v_{k_d})\]
for some $1\leq j_1\leq \dots \leq j_c\leq a$ and $1\leq k_1\leq \dots \leq k_d\leq b$.

By Lemma~\ref{lem:every b contains some a} we can form the $s$-fold product
$$P=(x_{n-1}u_{j_1})\dots (x_{n-1}u_{j_c})(x_{n-1}u_{k_1'})\dots (x_{n-1}u_{k_d'}) $$
where $u_{k_i'}$ divides $x_{n-2}v_{k_i}$ for each $i=1,\dots ,d$. By definition of $>_\MR$ we now have
\[P>_\MR Y_\ell >_\MR Y_r \, \text{ in } F(J(P_n)^s).\]
Lemma~\ref{lem:product of all As} implies that
\[u_{j_1}\dots u_{j_c}u_{k_1'}\dots u_{k_d'} \, >_\MR \,  u_{i_1}\dots u_{i_s} \, \text{ in } F(J(P_{n-2})^s).\]
Observe that 
	\begin{equation*}
\begin{split}
P\in G(J(P_n)^s)& \Longrightarrow u_{j_1}\dots u_{j_c}u_{k_1'}\dots u_{k_d'}=U_{t'} \, \text{ for some } t'<t \text{ by Lemma~\ref{lem:product of all As}}  \\
 & \Longrightarrow P:Y_r \in (U_1,\dots ,U_{t-1}):(U_t) \text{ as } P:Y_r=U_{t'}:U_t\\
 & \Longrightarrow Y_\ell: Y_r\in (U_1,\dots ,U_{t-1}):(U_t) \text{ as } P:Y_r \text{ divides }Y_\ell:Y_r
\end{split}
\end{equation*} 
as desired. On the other hand, if $P \notin G(J(P_n)^s)$, then by Lemma~\ref{lem:crucial}, there exists  $Y_\alpha \in G(J(P_n)^s)$ such that $Y_\alpha|P$ and $Y_\alpha>_{\MR} P$. Since $Y_\alpha|P$ it follows from Remark~\ref{rk:divisor of the same type} that $x_{n-1}^s|Y_{\alpha}$. Since $Y_\alpha>Y_r$ by Lemma~\ref{lem:product of all As} we get $Y_\alpha:Y_r\in(U_1,\dots ,U_{t-1}):(U_t)$. Since $Y_\alpha:Y_r$ divides $P:Y_r$ and $P:Y_r$ divides $Y_{\ell}:Y_r$, we have $Y_\alpha:Y_r$ divides $Y_{\ell}:Y_r$ and $Y_{\ell}:Y_r\in (U_1,\dots ,U_{t-1}):(U_t)$ as desired.

\medskip

%---------------------------------------------------
%                  CASE 3
%---------------------------------------------------

\emph{Case 3:} Suppose that $Y_r$ is divisible by $x_n x_{n-1}$ and it has the maximal expression
 $$Y_r=(x_{n-1}u_{i_1})\dots (x_{n-1}u_{i_q})(x_nx_{n-2}v_{j_1})\dots (x_nx_{n-2}v_{j_k})$$
 for some $1\leq i_1\leq \dots \leq i_q\leq a$ and $1\leq j_1\leq \dots \leq j_k\leq b$.
First note that from Proposition~\ref{prop:smart claim generalized} we have
\begin{equation}\label{eq: case3 main theorem}
x_{n-1}\in (Y_1,\dots ,Y_{r-1}):(Y_r).
\end{equation}
 Let $t<r$. Since $Y_t>_\MR Y_r$ and because of \eqref{eq: case3 main theorem} we may assume that $Y_t$ has the maximal expression
$$Y_t=(x_{n-1}u_{\alpha_1})\dots (x_{n-1}u_{\alpha_{q'}})(x_nx_{n-2}v_{\beta_1})\dots (x_nx_{n-2}v_{\beta_{k'}}) $$
for some $1\leq \alpha_1 \leq \dots \leq \alpha_{q'}\leq a$ and $1\leq \beta_1\leq \dots \leq \beta_{k'}\leq b$ with $q'\leq q$. We will now consider the following cases.

\emph{Case 3.1:} Suppose that $i_\ell = \alpha_\ell$ for all $\ell=1,\dots ,q'$. Then $q'=q$ since $Y_t>_\MR Y_r$. This implies $k=k'$. By Lemma~\ref{lem:mixed product maximal expression and minimal generator} we get $v_{\beta_1}\dots v_{\beta_k}>_\MR v_{j_1}\dots v_{j_k}$ in $\MR(J(P_{n-3})^k)$. Observe that 
\[Y_t:Y_r = v_{\beta_1}\dots v_{\beta_k} : v_{j_1}\dots v_{j_k}.\]
By the induction assumption on $J(P_{n-3})^k$ there exists a variable $x_z$ such that
\[x_z\text{ divides }v_{\beta_1}\dots v_{\beta_k} : v_{j_1}\dots v_{j_k} \text{ and } x_z=v_{\gamma_1}\dots v_{\gamma_k}:v_{j_1}\dots v_{j_k}\]
for some $v_{\gamma_1}\dots v_{\gamma_k}>_\MR v_{j_1}\dots v_{j_k}$ in $\MR(J(P_{n-3})^k)$. Therefore it suffices to show that
\[x_z\in (Y_1,\dots ,Y_{r-1}):(Y_r).\]
Consider the $s$-fold product $P=(x_{n-1}u_{i_1})\dots (x_{n-1}u_{i_q})(x_nx_{n-2}v_{\gamma_1})\ldots(x_nx_{n-2}v_{\gamma_k}).$ By definition of rooted order $P>_\MR Y_r$. Clearly $P:Y_r=x_z$. If $P\in G(J(P_n)^s)$ nothing is left to show. Otherwise, the result follows from Lemma~\ref{lem:crucial}.

\emph{Case 3.2:} Suppose that there is a smallest index $\ell$ among $1,\dots ,q'$ such that $i_\ell \neq \alpha_\ell$. Since $Y_t>_\MR Y_r$ we have $i_\ell > \alpha_\ell$. Then according to Lemma~\ref{lem: rooted list gives linear quotients} there exists a variable in $(u_1,\dots ,u_{i_\ell -1}):(u_{i_\ell})$, say $x_z$, which divides $u_{\alpha_\ell}:u_{i_\ell}$. Note that $x_z\neq x_{n-2}$ because of recursive definition of $\MR(P_{n-2})$. Also, it is clear that $x_z\neq x_n, x_{n-1}$ because $x_z$ is a vertex of $P_{n-2}$. From Proposition~\ref{prop:smart claim generalized} we see that
\[x_z\in (x_{n-1}u_1,\dots ,x_{n-1}u_{i_\ell -1}):(x_{n-1}u_{i_\ell})\subseteq (Y_1,\dots ,Y_{r-1}):(Y_r) \]
and thus it suffices to show that $x_z$ divides $Y_t:Y_r$. From Lemma~\ref{lem:trivial side of reduction characterization thm} and Proposition~\ref{prop:condition for non-minimal gen or non-maximal expression} we see that 
\[x_z\nmid (x_{n-1}u_{i_\ell})(x_{n-1}u_{i_{\ell+1}})\dots (x_{n-1}u_{i_q})(x_nx_{n-2}v_{j_1})\dots (x_nx_{n-2}v_{j_k}).\]
By the choice of $\ell$ since $u_{i_1}\dots u_{i_{\ell-1}}=u_{\alpha_1}\dots u_{\alpha_{\ell-1}}$ the result follows.
\end{proof}

Using Theorem~\ref{thm:main thm} one can obtain an exact formula for the regularity of powers of $J(P_n)$ as in the next corollary.

\begin{Corollary} For any $n\geq 2$ and $s\geq 1$
	\begin{equation*}
	\reg(J(P_n)^s)=
	\begin{cases*}
	2ks & if $n=3k+1$ or $n=3k$ \\
	2ks+s        & if $n=3k+2$.
	\end{cases*}
	\end{equation*}
\end{Corollary}
\begin{proof}
	Similar to proof of \cite[Corollary~5.3]{second power}.
\end{proof}

\section{Rooted order for chordal graphs}\label{sec: rooted order on chordal graphs}
In this section we will see how to generalize the concept of rooted list to chordal graphs. To simplify the notation we will use a set $A$ of vertices of $G$ interchangeably with the squarefree monomial $\prod_{x_i\in A}x_i$. 
\begin{Notation}
	For each $i=1,\dots ,r$ let $L_i$ be the list $L_i=a_1^i,\dots ,a_{k_i}^i$. Then by $L= L_1,L_2,\dots ,L_r$ we denote a new list $L$ which is obtained by joining the lists in the given order. More precisely,
	\[L=a_1^1,\dots ,a_{k_1}^1, a_1^2,\dots ,a_{k_2}^2,\cdots ,a_1^r,\dots ,a_{k_r}^r.\]
\end{Notation}

\begin{Definition}[\textbf{Rooted list for chordal graphs}]\label{def: rooted list for chordal}
	Suppose that $G$ is a chordal graph with a simplicial vertex $x_1$ such that $N[x_1]=\{x_1,\dots,x_m\}$ for some $m\geq 2$. We say $\MR(G)$ is a \emph{rooted list} of $G$ if it can be written in the form
		$$ \MR(H_1)N(x_1),\MR(H_2)N(x_2),\dots ,\MR(H_m)N(x_m)$$
	 where the list $\MR(H_i)$ is a rooted list of the subgraph $H_i=G\setminus N[x_i]$ for each $i=1,\dots, m$. If $G$ has no edges, then we set $\MR(G)=1$.
\end{Definition}

\begin{Remark}
	Observe that one can construct rooted lists in different ways as they depend on the choice of simplicial vertex. In Definition~\ref{def:rooted list} we always picked the last vertex $x_n$ of $P_n$ as a simplicial vertex.
\end{Remark}
\begin{Lemma}\label{lem:rooted list for chordal gives linear quotients}
		Let $G$ be a chordal graph with a rooted list $\MR(G)=u_1,\dots ,u_q$. Then
	\begin{enumerate}
		\item $G(J(G))=\{u_1,\dots ,u_q\}$
		\item $J(G)$ has linear quotients with respect to $u_1,\dots ,u_q$.
	\end{enumerate}
\end{Lemma}
\begin{proof}
	Proof follows from \cite[Theorem~3.1]{erey} and \cite[Theorem~2.13]{vv}.
\end{proof}

\begin{Definition}[\textbf{Rooted order/list for powers}]\label{def:rooted for powers of chordal} Let $G$ be a chordal graph with a rooted list $\MR (G)=u_1,\dots ,u_q$. We define a total order $>_{\MR}$ on $F(J(G)^s)$ which we call \emph{rooted order} as follows. For $M,N\in F(J(G)^s)$ with maximal expressions $M=u_1^{a_1}\dots u_q^{a_q}$ and $N=u_1^{b_1}\dots u_q^{b_q}$ we set $M>_{\MR} N$ if $(a_1,\dots ,a_q)>_{lex} (b_1,\dots ,b_q)$. 
	
Let $G(J(G)^s)=\{U_1, \ldots, U_r\}$. Then we say $U_1,\dots ,U_r$ is a \emph{rooted list} of minimal generators of $J(G)^s$ if $U_1>_{\MR} \ldots >_{\MR} U_r$. In such case, we denote the rooted list of generators by $\MR(J(G)^s)=U_1, \ldots, U_r.$ 
\end{Definition}

The following lemma is a version of Proposition~\ref{prop:smart claim generalized}.
\begin{Lemma}\label{lem: smart claim for chordal}
	Let $G$ be a chordal graph with $F(J(G)^s)=G(J(G)^s)$. Let $\MR(G)=u_1,\dots,u_q$ and let $\MR(J(G)^s)=Y_1,\dots,Y_p$. Suppose that $Y_r=u_{j_1}\dots u_{j_s}$ is the maximal expression for some $2\leq r\leq p$ with $j_1\leq \dots \leq j_s$. For each $1\leq t \leq s$ with $2\leq j_t$ we have
		\[(u_1,u_2,\dots ,u_{j_t-1}):(u_{j_t})\subseteq (Y_1,\dots ,Y_{r-1}):(Y_r).\]
\end{Lemma}
\begin{proof}
	  The ideal $(u_1,u_2,\dots ,u_{j_t-1}):(u_{j_t})$ is generated by variables since the rooted order gives linear quotients by Lemma~\ref{lem:rooted list for chordal gives linear quotients}. Let $\ell <j_t$ with $u_\ell:u_{j_t}=x_z$ for some variable $x_z$. Consider the $s$-fold product $M=Y_ru_\ell /u_{j_t}$. Then $M:Y_r=x_z$ and $M\root Y_r$. By assumption $M$ is a minimal generator of $J(G)^s$ and the proof follows.
\end{proof}

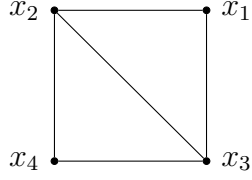
\begin{figure}[hbt!]
	\centering
	\begin{tikzpicture}
	[scale=1.00, vertices/.style={draw, fill=black, circle, minimum size = 4pt, inner sep=0.5pt}, another/.style={draw, fill=black, circle, minimum size = 2.5pt, inner sep=0.1pt}]
	\node[another, label=left:{$x_2$}] (b) at (-1, 1) {};
	\node[another, label=left:{$x_4$}] (c) at (-1, -1) {};
	\node[another, label=right:{$x_3$}] (e) at (1,-1) {};
	\node[another, label=right:{$x_1$}] (d) at (1,1) {};
	\foreach \to/\from in {b/c, b/e, c/e, b/d, d/e}
	\draw [-] (\to)--(\from);
	\end{tikzpicture}
	\caption{\label{diamond} Diamond graph}
\end{figure}

In the next chordal example, we construct a rooted list $\MR(G)$ such that rooted order $>_\MR$ on the generators of $J(G)^s$ yield linear quotients for all $s\geq 1$.
\begin{Example}\label{ex:diamond}
	Let $G$ be the chordal graph in Figure~\ref{diamond}. As in notation of Definition~\ref{def: rooted list for chordal} the vertex $x_1$ is a simplicial vertex and $N(x_1)=\{x_2,x_3\}$. Observe that $H_1$ is the graph consisting of the isolated vertex $x_4$. Also, $H_2$ and $H_3$ are empty graphs. Therefore we take $\MR(H_1)=\MR(H_2)=\MR(H_3)=1$. Then the rooted list of $G$ is $\MR(G)=u_1,u_2,u_3$ where
	\[u_1=x_2x_3, u_2=x_1x_3x_4, u_3=x_1x_2x_4.\]
	It is not hard to see that $$F(J(G)^s)=G(J(G)^s)$$ and every $s$-fold product has a unique expression. Let $\MR(J(G)^s)=Y_1,\dots,Y_p$. Now we will show that $J(G)^s$ has linear quotients with respect to the order $Y_1,\dots,Y_p$. Suppose that $Y_r=u_1^\alpha u_2^\beta u_3^\gamma$ with $r\geq 2$. Consider the ideal $I$ defined by
		\begin{equation*}
	I=
	\begin{cases*}
	(x_2,x_3) & if  $\beta\neq 0$ and $\gamma\neq 0$ \\
	(x_3)       & if $\beta=0$ and $\gamma\neq 0$ \\
	 (x_2) & if $\beta\neq 0$ and $\gamma=0$
	\end{cases*}
	\end{equation*}
	
Since $r\geq 2$, we claim that $I=(Y_1,\dots ,Y_{r-1}):(Y_r)$. It is clear from Lemma~\ref{lem: smart claim for chordal} that $I\subseteq (Y_1,\dots ,Y_{r-1}):(Y_r)$ because $(u_1):(u_2)=(x_2)$ and $(u_1,u_2):(u_3)=(x_3)$. To see the reverse, assume for a contradiction there exists $\ell <r$ such that no variable in $I$ divides $Y_{\ell}:Y_r$. Let $Y_\ell=u_1^{\alpha '} u_2^{\beta '} u_3^{\gamma '} $.

	\emph{Case 1:} Suppose $\beta\neq 0$ and $\gamma\neq 0$. Comparing exponents of $x_2$ and $x_3$ in $Y_\ell$ and $Y_r$ we see that $\alpha '+\gamma '\leq \alpha+\gamma$ and $\alpha '+\beta ' \leq \alpha+\beta $. Since both $Y_\ell$ and $Y_r$ are $s$-fold products we have $\alpha+\beta+\gamma = \alpha '+\beta '+\gamma '$ and thus $\alpha '\leq \alpha$. Since $Y_\ell >_\MR Y_r$ by definition of rooted order we get $\alpha '=\alpha$. This implies $\beta=\beta '$ and $\gamma=\gamma '$ and $\ell=r$, contradiction.

	\emph{Case 2:} Suppose $\beta=0$ and $\gamma\neq 0$. Comparing exponents of $x_3$ in $Y_\ell$ and $Y_r$ we see that $\alpha '+\beta '\leq \alpha$. In particular $\alpha ' \leq \alpha$. By definition of rooted order $\alpha'=\alpha$ must hold. This implies $\beta '=0$ and $\gamma '=\gamma$. Therefore $\ell=r$, contradiction.
	
	\emph{Case 3:} Suppose $\beta\neq 0$ and $\gamma=0$. Comparing exponents of $x_2$ in $Y_\ell$ and $Y_r$ we see that $\alpha '+\gamma '\leq \alpha$. In particular $\alpha ' \leq \alpha$. By definition of rooted order $\alpha'=\alpha$ must hold. This implies $\gamma '=0$ and $\beta '=\beta$. Therefore $\ell=r$, contradiction.
\end{Example}

We do not know any example of a power of a chordal graph which does not give linear quotients with respect to a rooted order. Therefore this led us to the following question.

\begin{question}
	Given a chordal graph $G$, does there exist a rooted list $\MR(G)$ such that the rooted order $>_\MR$ on the minimal generating set of $J(G)^s$ yields linear quotients for every $s\geq 1$?
\end{question}

\section*{Acknowledgment}
The author's research was partially supported by T\"{U}B\.{I}TAK, grant no. 118C033. We would like to thank the anonymous referee for his/her helpful comments. After the
submission of this paper, taking an entirely different approach, Herzog, Hibi and Moradi \cite{hhm} independently proved that all powers of the vertex cover ideal of a path graph have linear quotients.

\end{document}